\documentclass[10pt]{amsart}
\usepackage{a4,amssymb,times,amsmath,cancel}
\usepackage{graphicx,color,epsfig}

\usepackage[all]{xy}
\usepackage{tikz}
\usetikzlibrary{calc}

\def\ot{\leftarrow}

\let\cal\mathcal

\def\cC{{\cal C}}
\def\cD{{\cal D}}

\def\cI{{\cal I}}
\def\cJ{{\cal J}}

\def\cP{{\cal P}}
\def\cQ{{\cal Q}}
\def\cR{{\cal R}}
\def\cS{{\cal S}}

\def\cU{{\cal U}}

\def\cX{{\cal X}}
\def\cY{{\cal Y}}

\def\pZ{{\cal Z}}
\def\eps{{\epsilon}}

\let\blb\mathbb

\def \PP{{\blb P}}

\def \Rl{{\blb R}}

\newcommand{\se}[1]{\begin{equation*}\begin{split}#1\end{split}\end{equation*}}

\newcommand{\C}{\mathbb{C}}
\newcommand{\N}{\mathbb{N}}
\newcommand{\Z}{\mathbb{Z}}
\newcommand{\R}{\mathbb{R}}

\newcommand{\vtx}[1]{*+[o][F-]{\scriptscriptstyle #1}}

\newcommand{\Tr}{\textrm{Tr}}

\renewcommand{\>}{\rangle}

\newcommand{\Mod}{\ensuremath{\mathsf{Mod}}}

\newcommand{\simples}{\ensuremath{\mathsf{simples}}}
\newcommand{\End}{\ensuremath{\mathsf{End}}}

\newcommand{\Ext}{\mathsf{Ext}}
\newcommand{\Eqv}{\mathsf{Eqv}}

\font\ef = eufb10
\newcommand{\ideal}[1]{\ensuremath{\text{\ef{#1}}}}

\newcommand{\Span}{\mathrm{Span}}

\newtheorem{lemma}{Lemma}[section]

\newtheorem{theorem}[lemma]{Theorem}

\newtheorem{observation}[lemma]{Observation}

\theoremstyle{definition}

\newtheorem{example}[lemma]{Example}
\newtheorem{definition}[lemma]{Definition}

{

}

\theoremstyle{remark}

\newtheorem{remark}[lemma]{Remark}

\newcommand{\rep}{\ensuremath{\mathsf{rep}}}

\newcommand{\Mat}{\mathsf{Mat}}

\newcommand{\Coh}{\mathsf{Coh}}
\newcommand{\Hom}{\textrm{Hom}}

\newcommand{\GL}{\ensuremath{\mathsf{GL}}}

\newcommand{\qpol}{\cQ}
\newcommand{\carr}{\mathtt R}

\title{Consistency conditions for dimer models}

\author{Raf Bocklandt}
\address{Raf Bocklandt\\
School of Mathematics and Statistics\\
Herschel Building\\
Newcastle University\\
Newcastle upon Tyne\\
NE1 7RU\\
UK}
\email{raf.bocklandt@gmail.com}

\xyoption{all}

\begin{document}
\begin{abstract}
Dimer models are a combinatorial tool to describe certain algebras that
appear as noncommutative crepant resolutions of toric Gorenstein singularities.
Unfortunately, not every dimer model gives rise to a noncommutative crepant resolution. Several notions
of consistency have been introduced to deal with this problem. In this paper we study the major different notions in 
detail and show that for dimer models on a torus they are all equivalent.
\end{abstract}
\maketitle

\section{Introduction}\label{intro}
If $X$ is a 3-dimensional normal Gorenstein singularity admitting a crepant resolution $\tilde X\to X$, then one is interested
to describe the bounded derived category $\cD \Coh \tilde X$ of coherent sheaves on $\tilde X$. A well known result by Bridgeland \cite{bridgeland} shows that this category
only depends on the singularity and not on the particular choice of crepant resolution. 

In many cases there exists a tilting bundle in $\cU \in \cD \Coh \tilde X$ such that $\cD \Coh \tilde X$ is equivalent as a triangulated category
to the derived category of finitely generated $A$-modules $\cD \Mod A$ where $A=\End \cX$. 
To model these algebras without referring to a commutative crepant resolution, Van den Bergh \cite{nccrep} introduced the notion
of a a noncommutative crepant resolution (NCCR) of $X$. This is a homologically homogeneous algebra 
of the form $A=\End_R(T)$ where $T$ is a a reflexive $R$-module, with $R=\C[X]$ the coordinate ring of the singularity.
An NCCR is however far from unique and in general there are an infinite
number of different noncommutative crepant resolutions.

If we make two restrictions, the problem becomes more manageable. First we assume that $X$ is a \emph{toric} 3-dimensional Gorenstein singularity. This automaticly
implies the existence of a commutative crepant resolution. Secondly, we assume that the tilting bundle is a direct sum of non-isomorphic line bundles.
It was noticed in string theory \cite{HanKen,FranHan,HanHer} that under these conditions the algebra $A$ can be described using a dimer model on a torus.

This means that $A$ is the path algebra of a quiver $Q$ with relations where $Q$ is embedded in a 2 dimensional real torus $T$ such that every
connected piece of $T\setminus Q$ is bounded by a cyclic path of length at least 3. The relations are given by demanding
for every arrow $a$ that $p=q$ where $ap$ and $aq$ are the bounding cycles that contain $a$.

This nice description follows from the fact that the algebra $A$ is a toric order, 
a special type of order compatible with the toric structure, and 
Calabi-Yau-3 (CY-3), i.e. it admits a selfdual bimodule resolution of length $3$.
In \cite{bocklandtqp} it was shown that every toric CY-3 order comes from a dimer model.

Not every dimer model gives a noncommutative crepant resolution of its center. To do so,
it needs to satisfy some extra conditions called consistency conditions. In recent years, several quite different
consistency conditions were proposed such as cancellation \cite{davison}, non-intersecting zig and zag rays \cite{kenyon}, consistent R-charges \cite{Kennaway} and algebraic consistency \cite{broomhead}.

The aim of this paper is to show that for dimer models on a torus all these consistency conditions are equivalent.
Moreover the condition of being an order and the condition of being an NCCR are also equivalent to these consistency conditions.
The situation for the Calabi-Yau condition is less clear: it was shown by Davison \cite{davison} and Broomhead \cite{broomhead} 
that cancellation and algebraic consistency imply the CY-3 condition, but there is no proof for the other direction.
We will give an example of an infinite dimer model that is not cancellation but satisfies a suitable generalization
of the CY-3 property to the infinite case. However, no finite counterexamples are known.

If one broadens the definition of a dimer model to allow other compact surfaces,
the consistency conditions are no longer equivalent. We will also discuss the differences for those cases.

The paper is organized as follows. We start with a preliminary section on quivers and dimer models.
Then we introduce the cancellation property and discuss its relation with the CY-3 property.
In the other sections we introduce each time a notion of consistency and show that is equivalent to the cancellation property for dimer models
on a torus.
We end with a summary section which gives an overview of which equivalences hold for each type of surface.

\section{Acknowledgements}

I want to thank Nathan Broomhead, Alastair
King, Michael Wemyss, Travis Schedler and Sergei Mozgovoy for the interesting discussions
about these subjects at the Liegrits workshop in Oxford in January 2008 and during
my stays in Bath, Chicago and Wuppertal. I also want to thank Lieven Le Bruyn for
sharing his insights in the theory of reflexive Azumaya algebras.

\section{Preliminaries}

\subsection{Quivers}

A \emph{quiver} $Q$ is an oriented graph. We denote the set of vertices by $Q_0$, the set of arrows by $Q_1$ and the maps $h,t$ assign to each arrow its head and tail.
A \emph{nontrivial path} $p$ is a sequence of arrows $a_1\cdots a_k$ such that $t(a_i)=h(a_{i+1})$, whereas a \emph{trivial path} is just a vertex. We will denote the length of a path by $|p|:= k$ and the head and tail by $h(p)=h(a_1),~ t(p)=t(a_k)$. A path is called cyclic if $h(p)=t(p)$. 
Later on we will denote by $p[i]$ the $n-i^{th}$ arrow of $p$ and by $p[i\dots j]$ the subpath $p[i]\dots p[j]$.
\[
 \xymatrix{\vtx{}&\vtx{}\ar[l]_{p[n-1]}&\vtx{}\ar[l]_{p[n-2]}&\vtx{}\ar@{.>}[l]_{p[1]}&\vtx{}\ar[l]_{p[0]}}\text{ and }p = p[n-1]p[n-2]\dots p[1]p[0].
\]
A quiver is called \emph{connected} if it is not the disjoint union of two subquivers and it is \emph{strongly connected} if there is a cyclic path through each pair of vertices.

The \emph{path algebra} $\C Q$ is the complex vector space with as basis the paths in $Q$ and the multiplication of two paths $p$, $q$ is their concatenation $pq$ if $t(p)=h(q)$ or else $0$.
The span of all paths of nonzero length form an ideal which we denote by $\cJ$.
A \emph{path algebra with relations} $A=\C Q/\cI$ is the quotient of a path algebra by a finitely generated ideal $\cI \subset \cJ^2$. A path algebra is connected or strongly connected if and only if its underlying quiver is.
We will call a path algebra with relations $\C Q/\cI$ \emph{positively graded} if there exists a grading $\carr: Q_1 \to \Rl_{>0}$ such that $\cI$ is generated by homogeneous relations. 

\subsection{Dimer models}
A \emph{dimer model} $\qpol$ is a strongly connected quiver $Q$ enriched with a 2 disjoint sets of cycles of length at least $3$: $Q^+_2$ and $Q_2^-$, such that
\begin{itemize}
\item[DO] {\bf Orientability condition.} Every arrow is contained exactly once in one cycle in $Q_2^+$ and once in one in $Q_2^-$. 
\item[DM] {\bf Manifold condition.} The incidence graph of the cycles and arrows meeting a given vertex is connected.
\end{itemize}
The Euler characteristic of a dimer model is
\[
\chi_\qpol := \# Q_0-\# Q_1+ \# Q_2 \text{ with $Q_2=Q_2^+\cup Q_2^-$.} 
\]
From a dimer model we can build a topological space $|\qpol|$ by associating to every cycle of length $k$ in $Q_2$ a $k$-gon. 
We label the edges of this $k$-gon cyclicly by the arrows of the quiver and identify edges of different polygons labeled with the same arrow. 
If $\qpol$ satisfies DO and DM then $|\qpol|$ is a compact orientable surface with Euler characteristic $\chi_\qpol$ (\cite{bocklandtqp}).
We say that $\qpol$ is a dimer model on $|\qpol|$ and if $\chi_\qpol=0$ we speak of a dimer model on a torus.

To every dimer model we can associate its Jacobi algebra. For every $a \in Q_1$ we set $r_a = p^+-p^-$ where
$p^\pm a \in Q_2^\pm$  and then set 
\[
A_{\qpol}:= \C Q/ \<r_a | a\in Q_1 \>
\]
This algebra can be expressed in terms of a superpotential but we will not pursue this direction.

A dimer model is positively graded if there is a degree map $\carr:Q_1\to \R_{>0}$ such that all cycles
in $Q_2$ have the same degree. This turns $A_\qpol$ in a positively graded algebra.

Given a dimer model $\qpol$ on $X=|\qpol|$, we can look at the universal cover $\tilde X\to X$. We can lift the embedding of $\qpol$ in $X$
to obtain a possible infinite dimer model $\tilde \qpol$, which we will call the universal cover of $\qpol$.  

\begin{remark}
Usually the definition of a dimer model starts from the dual of $\qpol$: a bipartite graph on a surface, with nodes $Q_2$ and
edges $Q_1$, only after that the quiver is constructed by taking the dual. We don't do this because the switching can sometimes
cause confusion. We do keep the name dimer model, as it is most commonly used in the literature.
\end{remark}

\begin{example}
Below we give 4 examples of dimer models. the first is a dimer model on a sphere, the second and the third on a torus, the last on a double torus.
Arrows and vertices with the same label are identified.
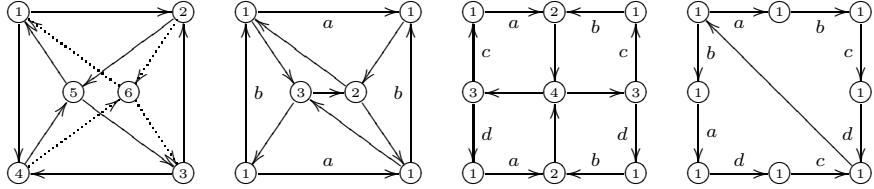
\[
\xymatrix@C=.4cm@R=.75cm{
\vtx{1}\ar[dd]\ar[rrr]&&&\vtx{2}\ar[dll]\ar@{.>}[dl]\\
&\vtx{5}\ar[ul]\ar[drr]&\vtx{6}\ar@{.>}[ull]\ar@{.>}[dr]&\\
\vtx{4}\ar[ur]\ar@{.>}[urr]&&&\vtx{3}\ar[uu]\ar[lll]
}
\hspace{.5cm}
\xymatrix@C=.4cm@R=.75cm{
\vtx{1}\ar[rrr]_a\ar[dr]&&&\vtx{1}\ar[ld]\\
&\vtx{3}\ar[r]\ar[ld]&\vtx{2}\ar[ull]\ar[dr]&\\
\vtx{1}\ar[rrr]^a\ar[uu]_b&&&\vtx{1}\ar[ull]\ar[uu]^b
}
\hspace{.5cm}
\xymatrix@C=.75cm@R=.75cm{
\vtx{1}\ar[r]_{a}&\vtx{2}\ar[d]&\vtx{1}\ar[l]^b\\
\vtx{3}\ar[u]_c\ar[d]^d&\vtx{4}\ar[l]\ar[r]&\vtx{3}\ar[u]^c\ar[d]_d\\
\vtx{1}\ar[r]^{a}&\vtx{2}\ar[u]&\vtx{1}\ar[l]_b
}
\hspace{.5cm}
\xymatrix@C=.75cm@R=.75cm{
\vtx{1}\ar[r]_{a}\ar[d]^{b}&\vtx{1}\ar[r]_b&\vtx{1}\ar[d]_c\\
\vtx{1}\ar[d]^a&&\vtx{1}\ar[d]_d\\
\vtx{1}\ar[r]^{d}&\vtx{1}\ar[r]^c&\vtx{1}\ar[uull]
}
\]
In the first example the cycles of $Q_2^+$ ($Q_2^-$) are the (anti-)clockwise triangles of the octahedron, in the last 3 examples the cycles $Q_2^+$ ($Q_2^-$) are (anti-)clockwise
triangles, quadrangles and pentagons.

Using the results of this paper, one can check that the first example is an order but not CY-3, 
the second nor an order nor CY-3, the third both an order and CY-3, the last CY-3 but not an order. 
The third is the only NCCR and its center is the coordinate ring over the cone over $\PP_1\times \PP_1$.
\end{example}

\section{Cancellation}

A path algebra of a quiver with relations is called a \emph{cancellation algebra} if for every arrow $a$ and any two paths $p,q$ with
$h(a)=t(p)=t(q)$ we have $pa=qa \implies p=q$ and for any two paths $p,q$ with
$t(a)=h(p)=h(q)$ we have $ap=aq \implies p=q$. 

For Jacobi algebras from dimer models we can restate the cancellation property.
The relations in the Jacobi algebra $A_{\qpol}$ imply that all cycles in $Q_2$ are equivalent: ${c_1}p=p{c_2}$ for every $p$ with $h(p)=t(c_1)$ and $t(p)=h(c_2)$. This implies that the algebra $A$ has a central element: $\sum c$ where we sum over a subset representatives of $Q_2$ that 
contains just one cyclic path $c$  with $h(c)=i$ for every $i\in Q_0$. 
We will denote this central element by $\ell$. 
For every arrow $a$ we can find a path $p$ such that $ap\in Q_2^+$ and hence $ap=h(a)\ell$ and $pa=t(a)\ell$.

The cancellation property states that the map
\[
 A_{\qpol} \to A_{\qpol} \otimes_{\C[\ell]} \C[\ell,\ell^{-1}]
\]
is an embedding. This tensor product is the algebra obtained by making every arrow invertible (i.e. for every $a$ we have an $a^{-1}$ such that $aa^{-1}=h(a)$ and
$a^{-1} a=t(a)$). This algebra is the localization of $A_{\qpol}$ by the Ore set $\{\ell^k|k \in \N\}$ and we denote it by $\hat A_{\qpol}$. 

We will take this property as the starting point from which we are going to prove the equivalences of the different consistency conditions.
But first we need to discuss the relation between cancellation and the Calabi-Yau property.
\section{The Calabi-Yau-3 condition}

\begin{definition}
A path algebra with relations $A$ is \emph{Calabi-Yau-3} 
if $A$ has a projective bimodule resolution $\cP^\bullet$ that is dual to its third shift
\[
 \Hom_{A-A}(\cP^\bullet,A\otimes A)[3] \cong \cP^\bullet
\]
\end{definition}
From this definition it is clear that a CY-3 algebra has global dimension $3$ and there are isomorphism
between $\Ext^i(X,Y)$ and $\Ext^{3-i}(X,Y)^*$ for every pair of finite dimensional left $A$-modules $X,Y$.
For more information about the CY-3 property we refer to \cite{GB} and \cite{bocklandt}.

That cancellation implies CY-3 was proved by Ben Davison in \cite{davison}. 
\begin{theorem}[Davison]\label{canc}
The Jacobi algebra of a dimer model $\qpol$ with $\chi_\qpol\le 0$ is CY-3 if it is a cancellation algebra.
\end{theorem}
\begin{remark}
Davison's work was a generalization of work by Mozgovoy and Reineke \cite{MR} which used an extra consistency condition. This extra condition turned out to be a consequence of the cancellation property.
\end{remark}

The method in the proof involved showing that a certain complex $\cC^\bullet$, which is by construction dual to its third shift, is acyclic.
This complex looks like
\[
\xymatrix{
\underset{s \in \cS}\bigoplus F_{s}\ar[r]^{\delta_3}&
\underset{r \in \cR}\bigoplus F_{r}\ar[r]^{\delta_2}&
\underset{b \in Q_1}\bigoplus F_{b}\ar[r]^{\delta_1}&
\underset{i \in Q_0}\bigoplus F_{i}
}
\]
where $\cR=\{r_a| a\in Q_1\}$ is the set of relations and $\cS=\{v\ell|v \in Q_0\}$. The bimodule $F_p$ is defined as
$Ah(p)\otimes p \otimes t(p)A$. The differentials have the following form:
\se{
\delta_1(q_1 \otimes b \otimes q_2) &= q_1b\otimes t(b) \otimes q_2 - q_1\otimes h(b)  \otimes bq_2\\ 
\delta_2(q_1 \otimes r \otimes q_2) &= \sum_k q_1a_1\dots \otimes a_{k} \otimes \dots a_n q_2 - 
\sum_k q_1b_1\dots \otimes b_{k} \otimes \dots b_m q_2\\
\delta_3(q_1 \otimes s \otimes q_2) &= \sum_{h(a)=h(s)} q_1a\otimes r_a \otimes 1 - \sum_{t(a)=h(s)} 1\otimes r_a \otimes a q_2.
}
In the second line we assume  $r=a_1\dots a_n - b_1\dots b_m$.

Let $G$ be the groupoid of paths in $\hat A_{\qpol}$, this groupoid gives a $G$-grading on every $F_p$ and 
this grading makes the complex homogeneous.

This complex can even be defined for infinite dimer models, but in that case the complex is
not dual to its third shift because taking the dual is not well-behaved for infinite direct sums.
We will still continue to call $A_\qpol$ CY-3 in the infinite case if $\cC$ is exact.

With this in mind we have the following observation:
\begin{observation}
There are infinite dimer models that are CY-3 but not cancellation.
\end{observation}
\begin{proof}(Sketch).
The example we give is
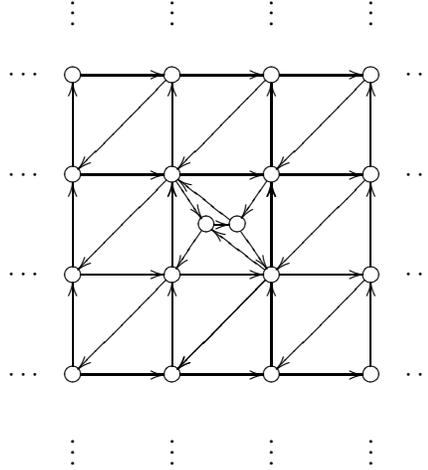
\[
\xymatrix@C=.2cm@R=.375cm{
&\vdots&&&\vdots&&&\vdots&&&\vdots&\\
\cdots&\vtx{}\ar[rrr]&&&\vtx{}\ar[rrr]\ar[ddlll]&&&\vtx{}\ar[rrr]\ar[ddlll]&&&\vtx{}\ar[ddlll]&\cdots\\
&&&&&&&&&&&\\
\cdots&\vtx{}\ar[rrr]\ar[uu]&&&\vtx{}\ar[rrr]\ar[ddlll]\ar[uu]\ar[dr]&&&\vtx{}\ar[ld]\ar[uu]\ar[rrr]&&&\vtx{}\ar[uu]\ar[ddlll]&\cdots\\
&&&&&\vtx{}\ar[r]\ar[ld]&\vtx{}\ar[ull]\ar[dr]&&&&&\\
\cdots&\vtx{}\ar[uu]\ar[rrr]&&&\vtx{}\ar[rrr]\ar[uu]\ar[ddlll]\ar[uu]&&&\vtx{}\ar[ddlll]\ar[uu]\ar[ull]\ar[uu]\ar[rrr]\ar[ddlll]\ar[uu]&&&\vtx{}\ar[uu]\ar[ddlll]&\cdots\\
&&&&&&&&&&&\\
\cdots&\vtx{}\ar[uu]\ar[rrr]&&&\vtx{}\ar[uu]\ar[rrr]&&&\vtx{}\ar[uu]\ar[rrr]&&&\vtx{}\ar[uu]&\cdots\\
&\vdots&&&\vdots&&&\vdots&&&\vdots&
}
\]
with all squares alike except the central one.
We can give all arrows degree $1$ to make $A_\qpol$ graded.

If $\cC^\bullet$ were not a bimodule resolution then either there would be extra syzygies. 
Because of the gradedness, there would be at least one simple $S_v:= A_{\qpol}v/(A_{\qpol}v)_{> 0}$ for which $\cC^\bullet\otimes_{A_{\qpol}} S_v$ is not a projective
resolution of $v$.

Now $\cC^\bullet\otimes_{A_{\qpol}} S_v$ can be decomposed in its $G$-homogeneous components and for each of the paths $p$ in $G$ one can 
check easily (see \cite{notcancelnote}) that the $p$-homogeneous part $\cC^\bullet\otimes_{A_{\qpol}} S_v$ is finite dimensional and acyclic.
\end{proof}
It is still an open question whether in the finite case there are CY-3 algebras that are not cancellation.

\section{Zigzag paths}

Checking whether $A_{\qpol}$ is a cancellation algebra is not an easy task. Here we will introduce a combinatorial criterion that will enable us to check this property visually.
In order to do this we need the following theorem from \cite{bocklandtqp} relating a quiver polyhedron and its universal cover.
\begin{theorem}\label{cancelcover}
The Jacobi algebra $A_{\qpol}$ is a cancellation algebra if and only if $A_{\tilde \qpol}$ is a cancellation algebra.
\end{theorem}

We will now restrict ourselves to the case when the Euler characteristic is nonpositive, so 
from now on we can assume that $|\tilde \qpol|$ is a contractible manifold. 

We can split any given path $p$ in the universal cover into \emph{positive (negative) arcs}. These are maximal subpaths that are 
contained in a positive (negative) cycle. We will a call path \emph{positively (negatively) irreducible} if none of its positive (negative) arcs is the derivative of a positive (negative) cycle or a cycle.

A \emph{zigzag path} is an infinite length path $\pZ = \dots \pZ[2]\pZ[1]\pZ[0]\pZ[-1]\pZ[-2]\dots$ for which all positive
and negative arcs have length $2$. It is easy to see that there are exactly two zigzag paths
for which $\pZ[0]$ equals a given arrow $a$ ($\pZ[1]\pZ[0]$ is either a positive or a negative arc).
We denote these two zigzag paths by $\pZ^+_a$ and $\pZ^-_a$. The part of the zigzag path $\pZ^+_a$ 
starting from $a$, $\pZ^+_a[i]_{i \ge 0}$ is called the zig ray from $a$ and is denoted by $\vec \pZ^+_a$.
Similarly we denote the zag ray by $\vec \pZ^-_a$.
The notion of a zigzag path is based on work by Kenyon in \cite{kenyon} and Kenyon and Schlenker in \cite{rhombic}. 

Every zigzag path $\pZ$ is bounded by a positively and a negatively irreducible path consisting of the positive (negative) arcs $u_i$ such that $u_i \pZ[2i+1]\pZ[2i]$ is a positive (negative) cycle.

\begin{theorem}\label{zigzag}
If $\chi_\qpol\le 0$ then $\qpol$ is cancellation if and only if
for every arrow $a \in \tilde Q_1$ the following condition hold
\begin{itemize}
 \item[Z] $\vec \pZ^+_a$ and $\vec \pZ^-_a$ only intersect in $a$ i.e. 
\[\forall i,j> 0: \pZ^+_a[i] \ne \pZ^-_a[j] \]
(note that the zigzag paths can intersect but only in different directions (f.i. $\pZ^+_a[i] = \pZ^-_a[j]$ with $i>0$ and $j<0$)
\end{itemize}
\end{theorem}
\begin{remark}\label{Z1}
Condition Z also implies that a zigzag path cannot selfintersect. Indeed if there are selfintersecting zigzag paths, we can take $\pZ$ such that $a=\pZ[0]=\pZ[i]$ and the loop $\pZ[i-1]\dots \pZ[0]$ encompasses the smallest number of cycles. If $\pZ=\pZ_a^\pm$ then the zigzag path $\pZ_a^\mp$ has an arrow inside the loop and as it cannot make a smaller
loop it must enter and leave the loop and hence there the zig and the zag ray of some arrow in $\pZ$ intersect. 
\end{remark}
\begin{proof}
Because  $\chi_\qpol\le 0$, the universal cover $|\tilde \qpol|$ cannot have the topology of a sphere. Therefore we have a 
definition of the interior and the exterior of a cyclic path that does not selfintersect.
{\bf Condition Z is necessary.}\\
If $\pZ^+_a[i] = \pZ^-_a[j]$ and $i,j$ are both positive and minimal, we look at the paths $p_+,p_-$ that
are the irreducible paths accompanying $\pZ^+_a[i]\dots\pZ^+_a[0]$ and $\pZ^-_a[j]\dots \pZ^-_a[0]$ which lie in the exterior of
$\pZ^+_a[i]\dots\pZ^+_a[0](\pZ^-_a[j]\dots \pZ^-_a[0])^{-1}$. By cancellation, we must have that $p_+=p_-\ell^k$ or  $p_+\ell^k=p_-$ for some $k\ge 0$.
This is impossible because $p_+$ or $p_-$ is positively or negatively irreducible.
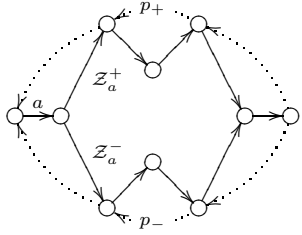
\[
\xymatrix@=.4cm{
&&\vtx{}\ar[rd]\ar@/_/@{.>}[lldd]&&\vtx{}\ar[rdd]\ar@/_/@{.>}[ll]|{p_+}&\\
&&&\vtx{}\ar[ru]&&\\
\vtx{}\ar[r]^a&\vtx{}\ar[ruu]_{\pZ^+_a}\ar[rdd]^{\pZ^-_a}&&&&\vtx{}\ar[r]&\vtx{}\ar@/_/@{.>}[lluu]\ar@/^/@{.>}[lldd]\\
&&&\vtx{}\ar[rd]&&\\
&&\vtx{}\ar[ru]\ar@/^/@{.>}[lluu]&&\vtx{}\ar[ruu]\ar@/^/@{.>}[ll]|{p_-}&
}
\]
{\bf Condition Z is sufficient.}\\
If $\qpol$ is not cancellation, we will assume that condition Z holds and search for a contradiction.

Let $p,q$ be paths such that in $A_\qpol$ $p\ne q$ but $p\ell^k=q\ell^k$ for some $k>0$.

A pair of paths $(p,q)$ with $h(p)=h(q)$ and $t(p)=h(q)$  
such that $pq^{-1}$ is a loop that does not selfintersect and its
their interior cannot be shrunk by applying the relations $r_a$ to $p$ or $q$, is called an irreducible pair.
It is clear that if $(p,q)$ is an irreducible pair then either $p$ or $q$ is a cycle in $Q_2$ or
one of the paths is negatively irreducible and the other one positively irreducible.

If the last case does not happen, $A_\qpol$ must be cancellation. 
Indeed, given two paths $p,q$ with the same head and tail, we can turn every loop in $p$ or $q$ to a product of cycles in $Q_2$.
After puting these cycles in front,
we can split $p=\ell^u p_1\dots p_k$ and $q=\ell^v q_1\dots q_l$ such that $p_i$ and $q_i$ coincide or do not intersect. 
By assumption $(p_i,q_i)$ is not irreducible so we can shrink it until the two paths coincide up to a power of $\ell$.
Hence, we can transform $p$ into $q\ell^k$ or $q$ into $p\ell^k$ for some $k\in \N$.

So let $(p,q)$ such that $p$ is a negatively irreducible path and $q$ a positively reducible path.

At $t(p)$ we consider the arrow $a$ with $h(a)=t(p)$ such that $a$ sits in the same negative cycle as the last positive arc of $p$.  The zigzag path $\pZ_a^+$ must enter $S$ at some vertex on the boundary. 

This vertex $v$ cannot lie on $p$. Indeed, if this
were the case it would be the head of one of the negative arcs of $p$ and $\pZ_a^+[j_1]=b$ with $j_1<0$ odd, $p= r_2br_1$ and $h(b)=v$. Now the negative cycle containing $b$ contains at least two arrows not in $p$ because $p$ is negatively reduced.
One of these arrows is $c=\pZ_a^+[j_1+1]$. Through $c$ we can look at the zigzag path $\pZ_c^-$. This path enters the  simply connected piece $S^{(1)}$ bounded by $p^{(1)}=br_1$ and $\pZ_a^+[0]\dots \pZ_a^+[j_1+1]$. It cannot leave $S^{(1)}$ through $\pZ_a^+$ by Z, so it must leave $S^{(1)}$ at a vertex which is the tail of some negative arc of $p$. This cannot be the first arc or otherwise the two zigzag paths intersect at $a$. This means that for some $j_2>0$ $\pZ_c^-[j_2-1]\dots\pZ_c^-[0]$ and a piece $p^{(2)}$ of $p$ cut out an even smaller piece $S^{(2)}$.
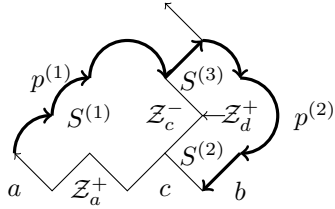
\begin{center}
\begin{tikzpicture}
\draw [<-] (0,0) -- (.5,-.5) -- (1,0) -- (1.5,-.5) -- (2,0) -- (2.5,-.5) -- (3,0);
\draw [->] (2,0) -- (2.5,.5) -- (2,1) -- (2.5,1.5) -- (2,2);
\draw [<-] (2.5,.5) -- (2.8,.5);
\draw [->,very thick] (0,0) arc (180:90:.5);
\draw [->,very thick] (.5,.5) arc (180:90:.5);
\draw [->,very thick] (1,1) arc (180:0:.5);
\draw [->,very thick] (2,1) -- (2.5,1.5);
\draw [->,very thick] (2.5,1.5) arc (90:0:.5);
\draw [->,very thick] (3,1) arc (90:-90:.5);
\draw [->,very thick] (3,0) -- (2.5,-.5);
\draw (0,-.5) node{$a$};
\draw (1,-.5) node{$\pZ_a^+$};
\draw (2,.5) node{$\pZ_c^-$};
\draw (3,-.5) node{$b$};
\draw (2,-.5) node{$c$};
\draw (1,.5) node{$S^{(1)}$};
\draw (2.5,0) node{$S^{(2)}$};
\draw (2.5,1) node{$S^{(3)}$};
\draw (.5,1) node{$p^{(1)}$};
\draw (4,0.5) node{$p^{(2)}$};
\draw (3,0.5) node{$\pZ_d^+$};
\end{tikzpicture}
\end{center}
Through $d=\pZ_c^-[j_2-1]$ we can look at the zigzag path $\pZ_d^-$ which cuts out an even smaller piece $S^{(3)}$.
If we continue this procedure we get to a point where $p^{(k)}$ has length zero. But this implies that the corresponding zigzag path will selfintersect (contradicting remark \ref{Z1}).

So $\pZ_a^+$ will leave $S$ through $q$. Because Z does hold, $\pZ_a^-$ must also leave through $q$. Analogously to the previous paragraph we can now construct a sequence of zigzag paths cutting a shorter and shorter pieces of $q'$ until we
get a selfintersecting zigzag path (contradicting remark \ref{Z1}).
\end{proof}
\begin{remark}
Dimer models with positive Euler characteristic can never satisfy the zigzag condition because
the universal cover is the quiver itself. As this quiver is finite, the zig and zag path ray intersect multiple times.
\end{remark}
\begin{remark}
A similar discussion on the relation between cancellation and zigzag paths was done by Ishii and Ueda in \cite{ishiicons} 
\end{remark}

\section{Consistent $\carr$-charges}
We borrow the following definition from Kennaway \cite{Kennaway}:
\begin{definition}
A grading $\carr: Q_1 \to \R_{>0}$ is consistent if 
\begin{itemize}
\item[R1]
$\forall c \in Q_2: \sum_{a \in c}\carr_a =2$,
\item[R2] 
$\forall v \in Q_0: \sum_{h(a)=v}(1-\carr_a) + \sum_{t(a)=v}(1-\carr_a)=2$.
\end{itemize}
\end{definition}
\begin{remark}
In \cite{broomhead} a distinction is made between geometrically consistent and marginally consistent $\carr$-charges. 
The former have the extra condition that $\carr_a<1$ for every $a$, while for the latter one also allows $\carr_a \ge 1$. 
We will not make this distinction: for us marginally consistent $\carr$-charges are also consistent.
\end{remark}

It has been pointed out in e.g. \cite{Kennaway} that
consistency implies that the Euler characteristic is zero
\se{
2\chi_{\qpol} &= 2(\#Q_2 - \#Q_1 + \#Q_0)\\ 
&= \sum_{c \in Q_2}\sum_{a \in c} \carr_a - \sum_a 2 + \sum_v \left(\sum_{h(a)=v}(1-\carr_a) + \sum_{t(a)=v}(1-\carr_a)\right)\\
&=\sum_a\left( \underbrace{2\carr_a}_{a \text{ sits in 2 cycles}} - 2 + \underbrace{1-\carr_a}_{v=h(a)}+ \underbrace{1-\carr_a}_{v=t(a)}\right)=0
}

Given a consistent $\carr$-charge we can realize the universal cover of the dimer model, which is the Euclidean plane, in the following way:
turn every cycle in $Q_2$ into a polygon the vertices of which are all on the unit circle and every arrow $a$ stands on an arc of $\pi \carr_a$ radians.
Condition R1 implies that such a polygon exists as the arcs add up to $2\pi$.
If $a$ and $b$ are consecutive arrows in a cycle then one can check that the angle between the two
arrows is $\frac \pi 2(2 - \carr_a -\carr_b)$ because it is the inscribed angle standing on the arc spanned by the rest of the cycle.
Pasting all these polygons together one gets a tiling of the plane because condition R2 implies that the angles of the polygons at one vertex add up to $2 \pi$ (see also \cite{rhombic}).
Such an embedding is called an \emph{embedding with isoradial cycles}\footnote{Such an embedding is a bit different from an isoradial embedding of the dimer model, which embeds the dual graph isoradially i.e.
the centers of the faces sharing a common vertex lie on a unit circle. Our definition also includes cases where the cycle does not contain the center of the circle (when one of the arrows has $\carr$-charge $\ge 1$)}.

A second ingredient we need are perfect matchings.
\begin{definition}
A perfect matching is a subset of arrows $\cP \subset \qpol_1$ such
that every cycle in $\qpol_2$ has exactly one arrow from $\cP$. A perfect matching gives a nonnegative grading on $A_\qpol$ by assigning degree $1$ to each arrow in $\cP$ and zero to the others:
\[
\deg_\cP a = \begin{cases}
1&a \in\cP\\
0&a\not \in \cP
\end{cases}
\]
\end{definition}

For an embedding with isoradial cycles we can construct special perfect matchings:
\begin{lemma}[Definition of $\cP_\theta^\pm$]
Given a $\qpol$ embedded with isoradial cycles and a direction $\theta$, then the set $\cP_\theta^+$  ($\cP_\theta^-$) of all arrows $a$ such that the ray from the center of its positive cycle in direction $\theta$ in the isoradial embedding hits $a$ but not in its head (tail), is a perfect matching. 
\end{lemma}
\begin{proof}
It is clear from the construction that every positive cycle has exactly one arrow in $\cP_\theta^\pm$.
The same holds for the negative cycles because
$a \in \cP_\theta^\pm$ if and only if the ray from the center of its negative cycle in direction $-\theta$ in the isoradial embedding hits $a$ but not in its tail (head). 
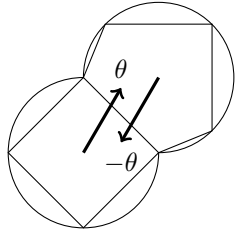
\begin{center}
\begin{tikzpicture}
\draw (1,0) -- (0,1) -- (-1,0) -- (0,-1) -- (1,0);
\draw (0,1) arc (90:360:1);
\draw (1,0) arc (-90:180:1);
\draw (0,1) -- (.29, 1.71) -- (1.71,1.71) -- (1.71,.29) -- (1,0);
\draw [->,very thick] (0,0) -- (.5,.86) node [above] {$\theta$};
\draw [->,very thick] (1,1) -- (.5,.14) node [below] {$-\theta$};
\end{tikzpicture}
\end{center}
\end{proof}

Now let $\pZ=\pZ_a^+$ be a zigzag path in a dimer model embedded with isoradial cycles.
We define $\eps_{\pZ} \in \Rl/2\pi\Rl$ to be the angle of $h(a)$ as viewed from the center
of the positive cycle containing $\pZ_a^+[1]\pZ_a^+[0]$.
It is easy to check that this definition does not depend on the $a$.
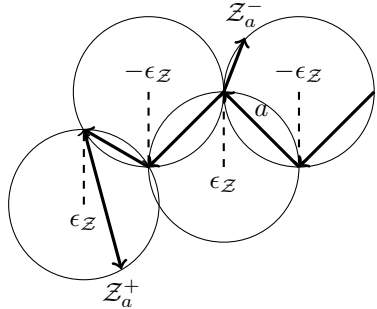
\begin{center}
\begin{tikzpicture}
\draw [->,very thick] (1,1) -- (0,0);
\draw [->,very thick] (0,0) -- (-1,1);
\draw [->,very thick] (-1,1) -- (-.72,1.72) node [above] {${\pZ_a^-}$};
\draw [->,very thick] (-1,1) -- (-2,0);
\draw [->,very thick] (-2,0) -- (-2.86,.5);
\draw [->,very thick] (-2.86,.5) -- (-2.36,-1.36) node [below] {${\pZ_a^+}$};
\draw (-.5,.75) node{$a$};
\draw (0,0) arc (0:360:1);
\draw (-1.86,-.5) arc (0:360:1);
\draw (-1,1) arc (0:360:1);
\draw (1,1) arc (0:360:1);
\draw [->,thick, dashed] (-1,0) node [below] {$\eps_{\pZ}$}  -- (-1,1);
\draw [->,thick, dashed] (-2,1) node [above] {$-\eps_{\pZ}$} -- (-2,0);
\draw [->,thick, dashed] (0,1) node [above] {$-\eps_{\pZ}$} -- (0,0);
\draw [->,thick, dashed] (-2.86,-.5) node [below] {$\eps_{\pZ}$} -- (-2.86,.5);
\end{tikzpicture}
\end{center}

\begin{lemma}\label{PMZZ}Let $\theta=\eps_{\pZ_a^\pm}$.
\begin{enumerate}
\item $\cP_\theta^\pm$ intersects $\pZ_a^\pm$ in all the arrows $\pZ_a^\pm[i]$ with $i$ odd.
\item Both $a$ and $\pZ_a^\mp[1]$ are not in $\cP_{\theta}^\pm$.
\end{enumerate}
\end{lemma}
\begin{proof}
We prove the statement for $\theta=\eps_{\pZ_a^+}$. As illustrated above,
viewed from the centers of the positive cycles, $t(\pZ_a^+[i])$ points in direction $\theta$ for all odd $i$.
Viewed from the center of the negative cycle, $t(a)$ points in the direction $-\theta$, so
the arrow $b=\pZ_a^+[-1]$ with head $t(a)$ must sit in $\cP_{\theta}^+$, this cannot be $\pZ_a^-[1]$ because
the cycle containing $ab$ has length at least $3$. 
\end{proof}

We are now ready to prove the equivalence between cancellation and the existence of a consistent $\carr$-charge.
\begin{theorem}\label{rcharge}
Let $\qpol$ be a dimer model on a torus then $\qpol$ is cancellation if and
only if it admits a consistent $\carr$-charge.
\end{theorem}
\begin{proof}
We will prove the equivalence of the existence of a consistent $\carr$-charge with property Z. After that we can apply theorem \ref{zigzag}.

{\bf The condition is sufficient.}
Suppose $\qpol$ has an $\carr$-charge and construct the corresponding tiling of the plane with isoradial cycles.
Suppose $\pZ_a^+[i]=\pZ_a^-[j]=b$ and let $p_+$ and $p_-$ be the positively and negatively irreducible paths in the opposite direction such that
$h(p_+)=h(p_-)=t(a)$ and $t(p_+)=t(p_-)=h(b)$.

If we take $\theta=\eps_{\pZ_a^+}$ then it follows from lemma \ref{PMZZ} 
$\deg_{\cP_\theta^+} p_+=0$. This implies that neither $a$ nor $\pZ_a^-[1]$ sit in
$\cP_\theta^+$. But $\cP_\theta^+$ must contain an arrow of the cycle through $\pZ_a^-[1]a$ so 
$\deg_{\cP_\theta^+} p_->0$. This means that $p_-=\ell^k p_+$ in $\hat A_\qpol$ for some $k>0$, so $\carr_{p_-}>\carr_{p_+}$.

On the other hand if we take $\theta=\eps_{\pZ_a^-}$ then for similar reasons
$\deg_{\cP_\theta^-} p_-=0$ but $\deg_{\cP_\theta^-} p_+>0$ and we get in $\hat A_\qpol$,  $p_+=\ell^l p_-$ for some $l>0$ and
$\carr_{p_-}<\carr_{p_+}$.
This contradicts the previous paragraph.

{\bf The condition is necessary.}
Every zigzag path $\pZ$ on the torus $|\qpol|$ is periodic and hence its lift $|\tilde \qpol|$ can be assigned a direction vector in the Euclidean plane. 
The unit vector in this direction will be denoted by $e_\pZ$. 

From condition Z, we can deduce that $e_{\pZ^+_a}\ne e_{\pZ^-_a}$ for every arrow $a$. If this were not
the case $\pZ^+_a$ and $\pZ^-_a$ would intersect an infinite number of times (in shifts of $a$ in the direction $e_{\pZ^-_a}$).

We now define an $\carr$-charge as $\frac 1\pi$ times the positive angle in clockwise direction from $e_{\pZ^-_a}$ to  $e_{\pZ^+_a}$
\[
\carr_a := \frac 1\pi \measuredangle(e_{\pZ^-_a}, e_{\pZ^+_a}). 
\]
The value of $\carr_a$ is always nonzero and smaller than $2$.

We now prove that the following two conditions hold:
\[
 \sum_{a \in c} \carr_a=2 \text{ and }\sum_{h(a)=v} (1-\carr_a)+ \sum_{t(a)=v} (1-\carr_a)=2.
\]
First look at the incidence structure of the zag rays starting from arrows around a positive cycle $c$ (i.e the $\vec \pZ^-_a$).
These rays do not intersect. If $a,b$ are consecutive arrows, the intersection of $\vec\pZ^-_a$ and $\vec\pZ^-_b$ would imply that
$\vec \pZ^+_a = a\vec\pZ^-_b$ and $\vec \pZ^-_a$ intersect twice.
If $a,b$ are not consecutive in the cycle, there must be an arrow $u$ between $a$ and $b$. But then $\vec\pZ_u^-$ must either intersect $\vec\pZ_a^-$ or $\vec\pZ_b^-$.
Proceeding in the same way, we can always find two consecutive arrows for which the zag rays intersect.

The non-intersection implies that the directions $e_{\pZ^-_a}$ are ordered on the unit circle in the same way as the arrows $a$. 
For consecutive arrows $a,b$ $e_{\pZ^+_a}=e_{\pZ^-_b}$ so the sum of the angles add up to $2\pi$ and hence the sum of the $\carr$-charges is $2$. 

We can now repeat this for the vertices. Look at all arrows leaving a vertex $v$. The zig rays of two consecutive leaving arrows do not intersect because otherwise 
we could follow the second zig path backwards inside the piece cut out by the two zig rays. This backwards path must leave this piece by an arrow $b$ of the first zig ray (because the second zag path cannot selfintersect). But now the zigzag rays from $b$ intersect twice, which contradicts Z.
If two zag rays of non-consecutive leaving arrows intersect then the zig ray of an arrow in between must intersect one of these zig rays so we can always reduce to the consecutive case. 

The angle between the directions of the zig rays of 2 consecutive leaving arrows $a_1,a_2$ is $\pi(2-\carr_{a_1} -\carr_b)$ where $b$ is the incoming arrow between $a_1$ and $a_2$. The fact that these angles add up to $2\pi$ gives us the second consistency condition.
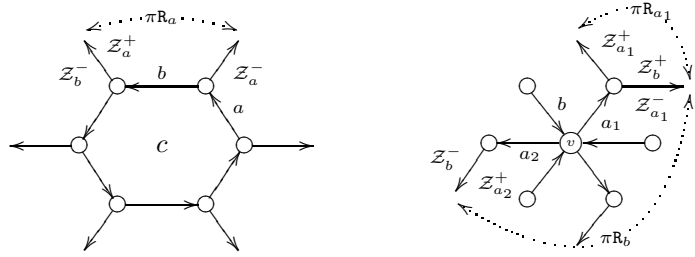
\[
\xymatrix@C=.3cm@R=.5cm{
&&\ar@{<-}[rd]_{\pZ_b^-}^{\pZ_a^+}&&&&\ar@{<-}[ld]^{\pZ_a^-}\ar@{<.>}@/_/[llll]|{\pi\carr_a}&&\\
&&&\vtx{}\ar[ld]&&\vtx{}\ar[ll]_b&&&\\
\ar@{<-}[rr]&&\vtx{}\ar[rd]&&c&&\vtx{}\ar[lu]_a&&\ar@{<-}[ll]\\
&&&\vtx{}\ar[rr]&&\vtx{}\ar[ru]&&&\\
&&\ar@{<-}[ru]&&&&\ar@{<-}[lu]&&\\
}
\hspace{1cm}
\xymatrix@C=.3cm@R=.5cm{
&&&&\ar@{<-}[rd]^{\pZ_{a_1}^+}&&&&\\
&&&\vtx{}&&\vtx{}\ar@{<-}[ld]^{a_1}&&\ar@{<-}[ll]^{\pZ_{a_1}^-}_{\pZ_b^+}\ar@{<.>}@/_20pt/[ulll]|{\pi\carr_{a_1}}&\\
&&\vtx{}\ar@{<-}[rr]_{a_2}&&\vtx{v}\ar@{<-}[lu]_b\ar@{<-}[rr]\ar@{<-}[ld]&&\vtx{}&&\\
&\ar@{<-}[ru]^{\pZ_b^-}_{\pZ_{a_2}^+}\ar@{<.>}@/_40pt/[uurrrrrr]|{\pi\carr_b}&&\vtx{}&&\vtx{}\ar@{<-}[lu]&&&\\
&&&&\ar@{<-}[ru]&&&&
}
\]
\end{proof}
\begin{remark}
The first part of this theorem is an extension of Lemma 5.3.1
in \cite{HanHer} to the marginally consistent case.
\end{remark}
\begin{remark}\label{goodtiling}
The second part of the theorem gives us an $\carr$-charge from the directions of the zigzag paths in the plane.
We can use this $\carr$-charge to construct an embedding with isoradial cycles. Note however that the angles between the zigzag paths in this new embedding are in general not the same as the ones we used to construct the $\carr$-charge.
We can recover these original directions from the embedding with isoradial cycles, because the $e_{\pZ_a^+}$ in the original embedding point precisely in the directions $\eps_{\pZ_a^+}$ of the new embedding.
\end{remark}

\section{Algebraic consistency}

In \cite{broomhead} Broomhead introduced the notion of algebraic consistency.
For this he constructed a second algebra from the dimer model: $B_\qpol$.
From $\qpol$ one can construct the following diagram of maps
\[
\Z \stackrel{e}{\ot} \Z^{Q_2} \stackrel{d}{\to} \Z^{Q_1} \stackrel{d}{\to} \Z^{Q_0} 
\]
with $e(c)=1$ and $d(c)=\sum_{a\in c} a$ for any cycle $c\in Q_2$ and $d(a)=h(a)-t(a)$.
We define $M=\Z^{Q_1}/de^{-1}(0)$ and for any vertices $i,j$ we set 
\[
M^+_{ij}=\frac{d^{-1}(i-j)\cap \N^{Q_1}}{de^{-1}(0)}.
\]
Then the $B$-algebra is defined as
\[
B_\qpol = \bigoplus_{i,j\in Q_0}\Span(M^+_{ij}) \subset \Mat_{|Q_0|}(\C[M])
\]
There is a natural map $\tau:A_\qpol\to B_\qpol: a \mapsto a \in \Span(M_{h(a)t(a)})$.
\begin{definition}
A dimer model is called algebraicly consistent if and only if $\tau:A_\qpol\to B_\qpol$ is 
an isomorphism.
\end{definition}

In \cite{broomhead} Broomhead proved: 
\begin{theorem}[Broomhead]
If a quiver polyhedron admits a geometrically consistent $\carr$-charge then it is algebraically consistent.
\end{theorem}
In this section we will extend this result to any consistent $\carr$-charge. The proof given will follow the same lines 
as Broomhead's proof. In particular we will use the following important lemma:
\begin{lemma}\cite{broomhead}[Broomhead Lemma 6.1.1]
If $\qpol$ is a dimer model on a torus then
$A_\qpol$ is algebraically consistent if it is cancellation and for every pair of vertices $v,w$ in the universal cover $\tilde \qpol$, there is a path $p:v \to w$ and a perfect matching $\tilde\cP$ that does not meet $p$. 
\end{lemma}
We also need these 3 lemmas.
\begin{lemma}\label{homotop}
Suppose $A_\qpol$ is cancellation and $\deg:Q_1 \to \R$ is any (not necessarily positive) grading such that $\deg\ell\ne 0$. 
Then two paths in $A_{\qpol}$ are equivalent if and only if they are homotopic and have the same $\deg$-degree.
\end{lemma}
\begin{proof}
It is clear that the relations $r_a$ imply that equivalent paths are homotopic and must have the same degree.
Because homotopies in the dimer model are generated by substituting paths $p\to q$ such that
$pq^{-1}=\ell$, homotopic paths can only differ by a factor $\ell^k$. The degree of $\ell$ is not zero, so
if homotopic paths have the same degree they must be equal in $A_{\qpol}$.
\end{proof}
\begin{lemma}\label{opp}
If $A_\qpol$ admits a consistent $\carr$-charge, $\cP$ is a perfect matching and $p$, $q$ are cyclic paths with opposite
homology classes then either $p$ or $q$ meets $\cP$.
\end{lemma}
\begin{proof}
Suppose $\deg_\cP p=\deg_\cP q=0$. 
Take any path $r$ from $h(p)$ to $h(q)$, then $\deg_\cP prq=\deg_\cP r$ and $prq$ and $r$ have the same homology
class, so by lemma \ref{homotop} they must be the same. But this is impossible because for the $\carr$-charge, $prq$ and $r$ must have different degrees.  
\end{proof}\label{angle}
\begin{lemma}
If $\qpol$ satisfies condition Z and $\chi_\qpol=0$ then
given a zigzag path $\pZ_1$ in the universal cover, there is always another zigzag path $\pZ_2$ making a positive angle with $\pZ_1$ less than $\pi$ radians: $0<\measuredangle(e_{\pZ_1}, e_{\pZ_2})<\pi$.
\end{lemma}
\begin{proof}
Suppose that the lemma does not hold. Let $\pZ_2$ be the zigzag path whose angle with $\pZ_1$ is smallest and
let $a$ be an arrow in their intersection. There are two possibilities:
$\pZ_1 = \vec \pZ_a^+$ and $\pZ_2 = \pZ_a^-$ or $\pZ_1 = \vec \pZ_a^-$ and $\pZ_2 = \pZ_a^+$.
In the first case, the directions in the zigzag paths show that there must be another arrow in the intersection behaving like the second case.
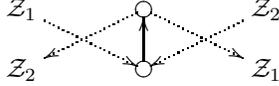
\[
\xymatrix@C=.5cm@R=.3cm{\pZ_1&&\vtx{}\ar@{.>}[lld]\ar@{.>}[rrd]&&\pZ_2\\
\pZ_2&&\vtx{}\ar[u]\ar@{<.}[llu]\ar@{<.}[rru]&&\pZ_1}
\]
So suppose $\pZ_1 = \vec \pZ_a^+$ and $\pZ_2 = \pZ_a^-$ and let $b=\pZ_a^+[-1]$.
By our hypothesis, the zigzag path $\pZ_b^+$ makes a positive angle with $\pZ_1$ that is at least as big as the angle with $\pZ_2$.
Therefore the $\pZ_b^+$ must intersect $\pZ_1$ a second time, but by condition Z $\pZ_b^+[i]=\pZ_1[j] \implies ij<0$. This also implies an intersection
of $\pZ_b^+$ with $\pZ_2$. This implies that $\pZ_2$ and $\pZ_b^+$ cannot have the same direction otherwise they would intersect multiple times in this direction (take shifts of the intersection). So $\pZ_b^+$ makes a positive angle with $\pZ_1$ that is bigger than the angle with $\pZ_2$. Now this implies a second intersection with $\pZ_2$ contradicting Z.
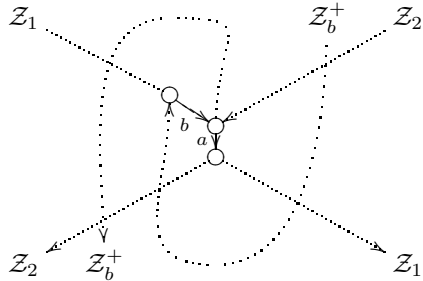
\[
\xymatrix@C=.4cm@R=.2cm{
\pZ_1\ar@{..}[rrdd]&&\ar@{.>}@(l,u)[lddddddd]&&&\pZ_b^+\ar@{..}@(d,r)[llddddddd]&\pZ_2\ar@{.>}[lllddd]\\
&&&&&&\\
&&\vtx{}\ar[rd]_b&&&&\\
&&&\vtx{}\ar[d]_a\ar@{..}@(u,r)[uuul]&&&\\
&&&\vtx{}\ar@{.>}[rrrddd]\ar@{.>}[lllddd]&&&\\
&&&&&&\\
&&&&&&\\
\pZ_2&\pZ_b^+&&\ar@{.>}@(l,d)[uuuuul]&&&\pZ_1\\
}
\]
\end{proof}

\begin{theorem}\label{algcon}
A dimer model with $\chi_\qpol=0$ is algebraically consistent if and only if it is cancellation.
\end{theorem}
\begin{proof}
Note that algebraic consistency automatically implies cancellation as $B_\qpol$ is a subalgebra of the cancellation algebra
$\Mat_n(\C[M])$.

Suppose that $\qpol$ is cancellation and let $0\ge \theta_1>\dots> \theta_u> 2\pi$ be the directions of the zigzag paths. 
Use these directions to construct an $\carr$-charge as in theorem \ref{rcharge} and its corresponding embedding with isoradial cycles.
For each $i \in \{1,\dots, u\}$ we define $\cP_i:= \cP_{\theta_i}^+$ (note that by the isoradial construction $\theta_i=\eps_{\pZ_i}$).

Every vertex in the universal cover has an arriving and a leaving arrow not in $\cP_{{i-1}} \cup \cP_{i}$.
Indeed if we look at the arrows in a vertex $v$ then by remark \ref{goodtiling} 
every arrow is a vector from $e_{\pZ_a^-}$ to $e_{\pZ_a^+}$, so the tail of each arriving arrow $a$
and the head of the leaving arrow $b$ in the same negative cycle are both in the same direction viewed from their
positive cycles. So if we shift all arrows arriving in and leaving from $v$ to the unit circle they will form a path
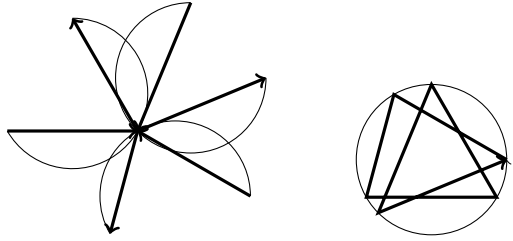
\begin{center}
\begin{tikzpicture}
\draw [->, very thick] ($(0:1cm)-(120:1cm)$) -- (0,0);
\draw [->, very thick] (0,0) --  ($(210:1cm)-(120:1cm)$); 
\draw [->, very thick] ($(210:1cm)-(330:1cm)$) -- (0,0); 
\draw [->, very thick] (0,0) -- ($(90:1cm)-(330:1cm)$); 
\draw [->, very thick]  ($(90:1cm)-(225:1cm)$) -- (0,0); 
\draw [->, very thick]  (0,0)-- ($(360:1cm)-(225:1cm)$); 
\draw ($(0:1cm)-(120:1cm)$) arc (0:210:1);
\draw ($(210:1cm)-(330:1cm)$) arc (-150:90:1);
\draw ($(90:1cm)-(225:1cm)$) arc (90:360:1);
\end{tikzpicture}
\hspace{1cm}
\begin{tikzpicture}
\draw [->, very thick] (0:1cm) --(120:1cm) -- (210:1cm) -- (330:1cm) -- (90:1cm) -- (225:1cm) -- (360:1cm);
\draw [->] (0:1cm) arc (0:360:1cm);
\end{tikzpicture}
\end{center}

The 2nd consistency condition implies that the path goes around
the unit circle $n-1$ times where $n$ is the number of incoming arrows. 
\[
\sum_{h(a)=v}(1-\carr_a) + \sum_{t(a)=v}(1-\carr_a)=2 \implies \sum_{h(a)=v}\carr_a + \sum_{t(a)=v}\carr_a=2n -2 
\]
An arrow sits in $\cP_{{i-1}} \cup \cP_{i}$ if and only if its head, tail or body crosses
the direction $\theta_i$. If all incoming arrows would cross $\theta_i$, the path would go round $n$ times which is a contradiction. The same can be said about the leaving arrows.

This means there is a path from every vertex $v$ that does not meet $\cP_{{i-1}} \cup \cP_{i}$ and hence does not intersects the zigzag path $\pZ_i$. It also does not selfintersect because it does not meet $\cP_{i}$.
Therefore it must either be parallel or antiparallel to the zigzag path. Parallel is impossible because of lemma \ref{opp} and the existence of a path in the opposite direction of the zigzag path. 

Let us call this ray $\cY_i^v$. If $\cY_i^v$ and $\cY_{i+1}^v$ intersect multiple times we know that the pieces between the intersections are equivalent
because they both do not meet $\cP_{i}$. Hence they also both do not meet $\cP_{i\pm 1}$. We can chose $\cY_i^v$ and $\cY_{i+1}^v$ to overlap
on that piece. Choosing the $\cY_i^v$ this way, we can divide the plane into sectors lying between the $\cY_i^v$ and $\cY_{i+1}^v$.

Now let $w$ be another vertex in the universal cover. If it lies on one of the rays $\cY_i$ then there is a path from $v$ to $w$ that does not meet $\cP_{i}$.
If it lies between $\cY_i^v$ and $\cY_{i+1}^v$ we can find a vertex $u_1$ far enough on $\cY_i^v$ and $u_2$ far enough on $\cY_{i+1}^v$ such that $w$ lies in the piece 
cut out by $\cY_i^v$,  $\cY_{i+1}^{u_1}$, $\cY_i^{u_2}$ and  $\cY_{i+1}^{v}$. Note that the middle two paths intersect
because by lemma \ref{angle} the angle in the original embedding between them is smaller than $\pi$.
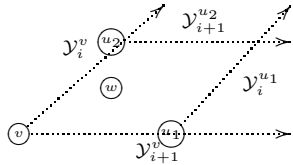
\[
\xymatrix@=.3cm{
&&&&&&&&\\
&&\vtx{u_2}\ar@{.>}[rrrr]^{\cY_{i+1}^{u_2}}&&&&&&\\
&&\vtx{w}&&&&&&\\
\vtx{v}\ar@{.>}[uuurrr]^{\cY_{i}^v}\ar@{.>}[rrrrrr]_{\cY_{i+1}^v}&&&\vtx{u_1}\ar@{.>}[uuurrr]_{\cY_{i}^{u_1}}&&&&&
}
\]

The piece is bounded by two paths that do not meet $\cP_{i}$ so they have the same homology and $\deg_{\cP_i}$ and by lemma \ref{homotop} they are equivalent.
Hence, there is a sequence of relations turning the first path in to the second. One of the intermediate steps must meet $w$ because it is inside the piece.
This will give us a path from $v$ to $w$ that does not meet $\cP_{i}$.
\end{proof}
\begin{remark}
The idea of cutting out a piece bounded by paths that do not meet a certain perfect matching is borrowed
from section 6.3.1 in \cite{broomhead}. In order to make this work in the marginal consistent case, we used the new notion
of these $\cP_\theta$ which do not appear in \cite{broomhead}. 
\end{remark}

\section{Orders}

\begin{definition}
An order $A$ is a prime algebra (i.e. the product of nonzero ideals is nonzero) 
which is a finitely generated module over its center $R$.
If $K$ is the quotient field of $R$ and $\Delta=A\otimes_R K$ then we say $A$ is an $R$-order in $\Delta$.  
\end{definition}

Orders have a special property: Reichstein and Vonessen \cite{reichstein} showed they can be reconstructed from a certain representation space.
Suppose $A$ can be written as a path algebra with relations $\C Q/\cI$. 

For any dimension vector $\alpha$ we can define $\rep(Q,\alpha)$ as
\[
\rep(Q,\alpha) := \bigoplus_{a \in Q_1}\Mat_{\alpha_{h(a)}\times \alpha_{t(a)}}(\C).
\]
This space parametrizes the $\alpha$-dimensional representations of $Q$.

On this space we have a base change action of the group $\GL_\alpha=\prod_{v\in Q_0}\GL_{\alpha_v}(\C)$.
This group also acts on $\Mat_{|\alpha|}(\C)$ by base change and we define $\Eqv(Q,\alpha)$ as the ring of equivariant
polynomial maps from $\rep(Q,\alpha)$ to $\Mat_{|\alpha|}(\C)$
\[
\Eqv(Q,\alpha) := \{f: \rep(Q,\alpha) \to \Mat_{|\alpha|}(\C)|\forall g \in \GL_\alpha: f(\rho^g) = f(\rho)^g\}
\]
The multiplication in this ring comes from $\Mat_{|\alpha|}(\C)$.

For $A$ we define the $\simples(A,\alpha)$ as the subset of $\rep(Q,\alpha)$ containing
all representations of $Q$ that are simple representation of $A$.

\begin{theorem}[Reichstein, Vonessen]\label{RV}
If $A$ is an order then there is an $\alpha$ such that 
\[
A \cong \frac{\Eqv(Q,\alpha)}{\{f: f(\simples(A,\alpha))=0\}}
\]
\end{theorem}
\begin{remark}
The original version of the theorem uses the terminology of PI-rings which is a bit broader than orders.
We also reformulated the theorem in the language of quivers whereas the original works with
generators of an algebra. The dimension vector is such that the PI-degree of $A$ is $|\alpha|$, this is the biggest $\alpha$ for which
there exist simples or equivalently the dimension vector of a simple representation of the form $\rho:A \to A\otimes_R R/\ideal m$ where
$\ideal m\lhd R$ is a maximal ideal. 
\end{remark}

In the case of dimer models we already had a notion of reconstructing the algebra $A$ using
the algebraic consistency. Algebraic consistency fits in this more general framework because of the following lemma
\begin{lemma}
If $\qpol$ is a dimer model then
$B_\qpol \cong \frac{\Eqv(Q,\alpha)}{\{f: f(\simples(A_{\qpol},\alpha))=0\}}$
with $\alpha$ the dimension vector that assigns $1$ to each vertex.
\end{lemma}
\begin{proof}
One can easily check that using the terminology of the previous section
\[
\Eqv(Q,\alpha) = \bigoplus_{i,j}\Span(d^{-1}(i-j)\cap \N^{Q_1}) \subset \Mat(\C[\Z^{Q_1}]).
 \]
If two monomials of $x^{\alpha},x^{\beta} \in \C[\Z^{Q_1}]$ evaluate the same on all simples,
then they evaluate the same on all representations for which all arrows are nonzero.
This implies that $\alpha-\beta \in de^{-1}(0)$ because such a simple representation is
a representation of $A_\qpol$ if and only if all cycles evaluate to the same number.
Therefore
\[
\frac {\Span(d^{-1}(i-j)\cap \N^{Q_1})}{\{f: f(\simples(A_{\qpol},\alpha))=0\}} = \Span(\frac{d^{-1}(i-j)\cap \N^{Q_1}}{de^{-1}(0)})
\]  
\end{proof}
So algebraic consistency seems just a specific consequence of being an order and indeed we have the following theorem.

\begin{theorem}
A Jacobi algebra of a dimer model on a torus is an order if and only if
it is algebraicly consistent. 
\end{theorem}
\begin{proof} 
An algebraicly consistent dimer model is always an order because $B_\qpol$ by construction
is prime and finite over its center.  If $A_\qpol$ is an order then it is cancellation because it is prime. Theorem \ref{algcon} implies it is algebraicly consistent. 
\end{proof}

\begin{remark}
A dimer model on a higher genus surface can never be an order, because
if it is prime, it must be cancellation, but then $\hat A_\qpol$ must also be finite over its center but this is impossible
because $\hat A_\qpol/(\ell-1)$ is Morita equivalent to the group algebra of a higher genus surface group.

A dimer model on a sphere is an order if and only if it is cancellation. Indeed if it is cancellation
and we are on a sphere then every path $p$ is isomorphic to a $\ell^k p'$ where $p'$ does not self intersect. 
There are only a finite number of paths that do not self intersect so $A_\qpol$ is finite over its center.
On the other hand, if $A_\qpol$ is an order then it is cancellation because it is prime.
Such dimer models are also algebraicly consistent: the representation $A_\qpol \to A_\qpol\otimes_{\C[\ell]}\C[\ell]/(\ell -1)=\Mat_{|Q_0|}(\C)$ so
the $\alpha$ from theorem \ref{RV} is the dimension vector that assigns $1$ to each vertex.
\end{remark}

\section{Noncommutative crepant resolutions}

In \cite{nccrep} Van den Bergh introduced the notion of a noncommutative crepant resolution.
\begin{definition}
Let $R$ be an affine commutative Gorenstein domain, with quotient field $K$. An algebra $A$ is a \emph{noncommutative crepant resolution}
of $R$ if $A$ is homologically homogeneous (i.e. the projective dimension of all simple representations of $A$ is the same)
and $A \cong \End_R(M)$ for some finitely generated $R$-reflexive module $M$  (reflexive means $\Hom_R(\Hom_R(M,R),R)\cong M$).
\end{definition}

As is explained in the discussion following this definition in and using results \cite{nccrep} from \cite{lebruynnote} and \cite{orzech}, this definition is satisfied if
\begin{enumerate}
 \item $A$ is an  $R$-order in $\Mat_{n\times n}(K)$,
 \item $A$ has finite global dimension,
 \item $A$ is Cohen-Macaulay over $R$,
 \item the ramification locus has codimension $\ge 2$.
\end{enumerate}
The ramification locus of an order is defined as the set of points $\ideal p \in \mathtt{Mspec} R$ such 
that $A \otimes_R R/\ideal p \ne \Mat_{n \times n}(\C)$ (or in other words the representation of $A$ at the point $p$ is not simple). 

\begin{theorem}
The Jacobi algebra of a dimer model on a torus is a noncommutative crepant resolution of its center if and only if it is cancellation.
\end{theorem}
\begin{proof}
If $A_\qpol$ is a noncommutative crepant resolution of its center then it is an order and hence cancellation.

Suppose that $A_\qpol$ is cancellation, then $A_\qpol$ is an order.
Because $A_\qpol\otimes_{\C[\ell]} \C[\ell,\ell^{-1}]=\Mat_n(\C[X_1^{\pm 1},X_2^{\pm 1},X_3^{\pm 1}])$ 
and  $K=\C(X_1,X_2,X_3)$ we have that $A_{\qpol} \otimes_R K = \Mat_{n\times n}(K)$.
By the CY-3 property the global dimension is $3$.
From \cite{StaffordVDB}[theorem 2.2] we conclude that $A$ is Cohen-Macaulay.

Finally, if we show that the ramification locus has codimension at least $2$, we are done.
If $\ideal p\lhd Z(A)$ lies in the ramification locus then $\ell$ must be in $\ideal p$, because
otherwise all arrows must evaluate to something nonzero and the representation is simple.

Now we show that there is at least one cycle with nonzero homology class that evaluates to zero:
if this were not the case we could find two cycles $c_1$,$c_2$ with linearly independent homotopy classes that
are not zero. Both cycles can be seen as $h(c_1)\Tr c_1$ and $h(c_2)\Tr c_2$ where $\Tr c_1,\Tr c_2 \in R$. 
If $v_1$ and $v_2$ are two vertices then we can look at $v_1\Tr c_1$ and $v_2\Tr c_2$.
These are two cycles, they are nonzero and because the homotopy classes are linearly independent, they must intersect.
This means that there is a path of nonzero arrows between $v_1$ and $v_2$. As this holds for
every couple of vertices, the representation must be simple.

Two zero cycles ($\ell$ and the one with nontrivial homology) with different homology generate an ideal which defines a subscheme of codimension $2$.
\end{proof}
\begin{remark}
A different proof of this statement can be found in \cite{broomhead}.
\end{remark}
\begin{remark}
A dimer model on a higher genus surface can never be an NCCR because it is not an order.
A dimer model on a sphere can never be an NCCR because even if it were an order, its center is $\C[\ell]$, which is smooth so the NCCR should be equal to $\C[\ell]$.
\end{remark}

\section{Summary}

The following theorem is a summary of all main theorems from the previous sections:
\begin{theorem}
For a dimer model $\qpol$ on a torus the following are equivalent:
\begin{enumerate}
\item $A_\qpol$ is cancellation.
\item $A_\qpol$ is algebraically consistent.
\item $A_\qpol$ is an order.
\item $A_\qpol$ is an NCCR of its center.
\item The zig and zag rays in the universal cover do not intersect twice.
\item There exists a consistent $\carr$-charge.
\end{enumerate}
\end{theorem}
The remarks following these proofs show that 
in the higher genus case this theorem changes to:
\begin{theorem}
For a dimer model $\qpol$ on a higher genus surface the following are equivalent:
\begin{enumerate}
\item $A_\qpol$ is cancellation.
\item The zig and zag rays in the universal cover do not intersect twice.
\end{enumerate}
While the following can never happen
\begin{enumerate}
\item There exists a consistent $\carr$-charge.
\item $A_\qpol$ is algebraically consistent.
\item $A_\qpol$ is an order.
\item $A_\qpol$ is an NCCR of its center.
\end{enumerate}
\end{theorem}
In the genus zero case we have got:
\begin{theorem}
For a dimer model $\qpol$ on a sphere the following are equivalent:
\begin{enumerate}
\item $A_\qpol$ is cancellation.
\item $A_\qpol$ is algebraically consistent.
\item $A_\qpol$ is an order.
\end{enumerate}
While the following can never happen
\begin{enumerate}
\item $A_\qpol$ is CY-3. 
\item The zig and zag rays in the universal cover do not intersect twice.
\item There exists a consistent $\carr$-charge.
\item $A_\qpol$ is an NCCR of its center.
\end{enumerate}
\end{theorem}

\bibliographystyle{amsplain}
\def\cprime{$'$}
\providecommand{\bysame}{\leavevmode\hbox to3em{\hrulefill}\thinspace}
\providecommand{\MR}{\relax\ifhmode\unskip\space\fi MR }
\providecommand{\MRhref}[2]{%
  \href{http://www.ams.org/mathscinet-getitem?mr=#1}{#2}
}
\providecommand{\href}[2]{#2}

\end{document}